\documentclass[12pt]{article}

\usepackage[margin=1in]{geometry}
\usepackage{amsmath,amsthm,amssymb,amsfonts}
\usepackage{mathtools}
\usepackage{graphicx}
\usepackage{hyperref}
\usepackage{url}
\usepackage[numbers,sort&compress]{natbib}
\usepackage{enumitem}
\usepackage{placeins}
\usepackage{silence}
\usepackage[expansion=false]{microtype}

\theoremstyle{plain}
\newtheorem{theorem}{Theorem}[section]
\newtheorem{proposition}[theorem]{Proposition}
\newtheorem{lemma}[theorem]{Lemma}
\newtheorem{corollary}[theorem]{Corollary}

\theoremstyle{definition}
\newtheorem{definition}[theorem]{Definition}

\newtheorem{counterexample}[theorem]{Counterexample}

\theoremstyle{remark}
\newtheorem{remark}[theorem]{Remark}

\newcommand{\R}{\mathbb{R}}
\newcommand{\doi}[1]{\textsc{doi:}~\texttt{#1}}
\DeclareMathOperator{\Crit}{Crit}
\DeclareMathOperator{\intr}{int}
\newcommand{\p}{\partial}
\newcommand{\Cred}{C_{\mathrm{red}}}
\newcommand{\Ssec}{\Sigma_E^{\mathrm{sec}}}
\newcommand{\Areg}{A_E}
\newcommand{\DE}{D_E}
\newcommand{\ME}{M_E}                      
\newcommand{\Ered}{\mathcal{E}}            
\newcommand{\hb}{h_{b}}
\newcommand{\Fb}{F_{b}}
\newcommand{\Eesc}{E_{\mathrm{esc}}}

\usepackage[nameinlink,capitalise]{cleveref}
\usepackage{orcidlink}

\title{Reconstruction through a coisotropic reduction:\\
	the missing momentum on a Poincar\'e section as a branched covering}
	
	\author{E. Chan-L\'opez\huge\textsuperscript{\orcidlink{0009-0003-7712-7907}}\thanks{E-mail: \texttt{eduardo.clopez13@gmail.com}.}}

\date{%
	Divisi\'on Acad\'emica de Ciencias B\'asicas,
	Universidad Ju\'arez Aut\'onoma de Tabasco, 86690 Cunduac\'an,
	Tabasco, Mexico}

\begin{document}
	\maketitle
	
	\begin{abstract}
		
		On a Poincar\'e section $\{q_{1}=0\}$ of a two-degree-of-freedom Hamiltonian
		system at energy $E$, the momentum $p_{1}$ is not a coordinate and must be
		recovered from the energy relation $H(0,q_{2},p_{1},p_{2})=E$. We organize
		this recovery geometrically. The section is a coisotropic hypersurface, $p_{1}$
		is the coordinate along its characteristic leaves, and the object to be
		studied is the projection $\pi_E$ of
		$\Sigma_E^{\mathrm{sec}}=\{H=E\}\cap\{q_{1}=0\}$ to the reduced plane
		$(q_{2},p_{2})$. Its critical points are exactly the points where the
		Hamiltonian flow is tangent to the section, and where the symplectic form
		induced on $\Sigma_E^{\mathrm{sec}}$ degenerates.
		
		We separate the hypotheses behind each feature of $\pi_E$.
		Fibrewise coercivity makes $\pi_E$ proper and, away from its discriminant
		$D_E$, a finite-sheeted covering of its image. Under unimodality, there are
		exactly two sheets $p^{\pm}$ over $\{\mathcal E<E\}$, where $\mathcal E$ is
		the fibrewise least energy, and the forward-crossing sheet is $p^{+}$.
		Fibrewise nondegeneracy and transversality at a critical point give a
		Whitney fold there. This covering-plus-fold structure is what we call a
		branched covering. The discriminant coincides with the boundary of the
		admissible region $A_E=\pi_E(\Sigma_E^{\mathrm{sec}})$ if and only if
		$\mathcal E$ has no local maximum with value $E$. Regularity of the level,
		which we show is equivalent to $E$ being a regular value of $H$ restricted
		to the section, is sufficient but not necessary. A positive-definite
		momentum Hessian with coercive fibres removes the obstruction at every
		energy, a condition satisfied by all natural mechanical systems. An
		example with nondegenerate coercive fibres and an interior discriminant
		point shows that fibrewise hypotheses alone do not suffice. Finally, the
		squared residual $F_b=(h_b-E)^{2}$ is a variational reformulation that
		makes the degeneration explicit: over $\{\mathcal E<E\}$ it is a double
		well with two global minimizers, so branch selection is a separate rule.
		The H\'enon--Heiles and swing--spring systems illustrate the construction.
	\end{abstract}
	
	\noindent\textbf{Keywords:} Poincar\'e section, coisotropic reduction,
	reconstruction, fold singularity, discriminant, admissible region, branched
	covering.
	
	\medskip
	\noindent\textbf{2020 Mathematics Subject Classification:}
	37J39, 70H33, 53D20, 58K05, 65P10.
	
	\section{Introduction}\label{sec:intro}
	
	Poincar\'e sections replace a Hamiltonian flow by a return map on a fixed
	surface and are a standard tool for exposing the structure of nonlinear
	dynamics \cite{zhangdai2025,elnabulsi2026}. Fixing an energy and intersecting
	the flow with a section leaves a preliminary geometric question implicit. The
	section keeps only part of the phase-space coordinates: the momentum conjugate
	to the eliminated direction must be recovered from the energy relation, and
	the recovery degenerates on a certain curve of the section. The word
	\emph{reconstruction} is used in more than one sense in this context; it
	denotes, for instance, the data-driven identification of a return map from
	samples on the section \cite{wang2026reconstruction}. We use it in a strictly
	geometric sense: the recovery of the eliminated coordinate from a point of the
	reduced section and the value of the energy.
	
	When a symplectic manifold is reduced along a coisotropic hypersurface, the
	coordinate tangent to the characteristic foliation is discarded by
	construction \cite{arnold,jose1998,abraham1978}. Recovering, from a point of
	the reduced space and the value of a constraint, a point of the original
	hypersurface is the inverse operation, and this is what we call
	reconstruction. Let $H\colon\R^{4}\to\R$ be a $C^{2}$ Hamiltonian, fix an
	energy $E$, and take the section $C=\{q_{1}=0\}$. Given $(q_{2},p_{2})$ one
	seeks $p_{1}$ from
	\begin{equation}\label{eq:constraint}
		H(0,q_{2},p_{1},p_{2})=E .
	\end{equation}
	The computation of section crossings is classical \cite{henon}, and
	\eqref{eq:constraint} is usually treated as a root-finding step attached to
	it. The step is ill conditioned where $\p H/\p p_{1}$ vanishes. Since
	$\dot q_{1}=\p H/\p p_{1}$, these are exactly the points where the flow is
	tangent to the section.
	
	The elementary phenomena attached to \eqref{eq:constraint} are not new. That
	projections of energy surfaces are two-to-one, that the two solutions
	coalesce on a tangency set, that a boundary of accessible states appears, and
	that this boundary reorganizes at critical values of an effective potential,
	are all familiar from rigid-body dynamics: the enveloping surfaces and
	admissible-velocity sets of Gashenenko and Richter \cite{gashenenko2004}, the
	tangent sets of Dullin and Wittek \cite{dullin1995}, the role of the
	accessible region in the topology of energy surfaces \cite{bolsinov1996}, and
	the Poincar\'e section of Schmidt, Dullin and Richter \cite{schmidt2009}, in
	which the flow is not transverse to the section and the section is cut along
	the tangency set and projected two-to-one. In natural mechanical systems the
	missing momentum is recovered by a square root, whose two branches are exactly
	this projection. (In configuration space the boundary of the accessible region
	is where turning points occur, and its dynamical role has been studied
	\cite{salasnich2026}; the boundary relevant here is a different one: it lives
	in the reduced section, and on it only $p_{1}$ vanishes.)
	
	What we add is an organization of these facts for a general Hamiltonian,
	with the hypotheses behind each of them kept apart. The question is
	\begin{quote}
		\emph{what controls, separately, the existence, the multiplicity, the
			regularity and the degeneration of the lift of $(q_{2},p_{2})$ to the
			energy level, through the reduction that has forgotten $p_{1}$?}
	\end{quote}
	The section $C$ is a coisotropic hypersurface of $(\R^{4},\omega)$ with
	characteristic direction $\R\,\p_{p_{1}}$, and the projection $\pi\colon
	C\to\Cred\cong\R^{2}$ that forgets $p_{1}$ is the associated reduction. The
	natural carrier of the recovery is the set $\Ssec=\Sigma_E\cap C$ together
	with $\pi_E=\pi|_{\Ssec}$. The coisotropic language does three concrete things
	here (Proposition~\ref{prop:flow}, Corollary~\ref{cor:sheets}): the critical
	points of $\pi_E$ are the tangencies of the flow with $C$; the form
	$\omega|_{\Ssec}$, which the return map preserves, equals
	$\pi_E^{*}(dq_{2}\wedge dp_{2})$ and degenerates exactly there; and, under
	unimodality, the forward-crossing part $\{\dot q_{1}>0\}$ of $\Ssec$ is one
	global sheet of $\pi_E$.
	
	The results are the following.
	
	\emph{(i) Covering and fold.} Coercivity (C) of the fibre functions
	$h_{b}=H(0,q_{2},\cdot\,,p_{2})$ makes $\pi_E$ proper and, over the complement of the discriminant
	$\DE$, a local diffeomorphism with a finite number of sheets that is constant
	on each component and may be zero (Theorem~\ref{thm:cover}): a finite-sheeted
	covering of its image. With unimodality
	(U) the count is two, one or zero according to the sign of $\Ered(b)-E$, where
	$\Ered(b)=\min_{p_{1}}h_{b}$, and the covering over $\{\Ered<E\}$ is trivial
	with sheets $p^{\pm}$ (Proposition~\ref{prop:mult},
	Corollary~\ref{cor:sheets}). At a critical point satisfying fibrewise
	nondegeneracy (N) and transversality (T), $\pi_E$ is a Whitney fold
	$(u,v)\mapsto(u,v^{2})$ (Theorem~\ref{thm:fold}). We call the resulting
	structure a \emph{fold-branched double covering}
	(Definition~\ref{def:branched}); this is the sense of ``branched covering'' in
	the title, and it is not the notion of low-dimensional topology.
	
	\emph{(ii) Discriminant versus boundary.} Under (C) one always has
	$\p\Areg\subseteq\DE$. Under (C) and (U),
	\[
	\DE=\p\Areg\iff \Ered \text{ has no local maximum on the level } E
	\]
	(Theorem~\ref{thm:boundary}); the points of $\DE$ that are not boundary
	points are precisely those local maxima. Regularity (R) of the level
	$\{\Ered=E\}$ is a sufficient condition, not a necessary one, and it is
	equivalent to $E$ being a regular value of $H|_{C}$, that is, to (T) at every
	critical point (Proposition~\ref{prop:reg}). If the momentum Hessian is
	positive definite and (C) holds, then $\Ered$ has no local maxima at all and
	$\DE=\p\Areg$ at every energy (Corollary~\ref{cor:convexbdry}); this covers
	all natural mechanical systems. Fibrewise hypotheses alone do not suffice:
	Counterexample~\ref{cex:hat} satisfies (C), (U) and (N) and has an interior
	discriminant point.
	
	\emph{(iii) What the least-energy function does not see, and the residual.}
	The function $\Ered$ determines $\Areg$ but not $\Ssec$
	(Section~\ref{ssec:loss}). The squared residual $\Fb=(\hb-E)^{2}$ reformulates
	the recovery variationally: it has a minimizer over every base point, two of
	them over $\{\Ered<E\}$, and its curvature at the minimizers vanishes linearly
	at the discriminant (Section~\ref{sec:variational}). It does not select a
	sheet and does not remove the degeneration.
	
	Section~\ref{sec:setting} fixes the geometric datum and its relation to the
	flow. Section~\ref{sec:hyp} states the hypotheses, the elementary properties
	of $\Ered$, and the exact role of momentum convexity.
	Section~\ref{sec:structure} contains the covering, the two sheets and the
	fold. Section~\ref{sec:boundary} compares the discriminant with $\p\Areg$, and
	Section~\ref{sec:interior} gives the interior-point example.
	Section~\ref{sec:variational} treats the residual,
	Section~\ref{sec:applications} the H\'enon--Heiles and swing--spring systems,
	and Section~\ref{sec:impl} the accompanying computations.
	
	\section{Geometric setting and the reconstruction correspondence}
	\label{sec:setting}
	
	\begin{definition}[Hamiltonian system and section]\label{def:ham}
		Let $(\R^{4},\omega)$ be the standard symplectic space with coordinates
		$(q_{1},q_{2},p_{1},p_{2})$ and $\omega=dq_{1}\wedge dp_{1}+dq_{2}\wedge
		dp_{2}$, and let $H\in C^{2}(\R^{4})$, with Hamiltonian vector field
		$X_{H}=(\p_{p}H,-\p_{q}H)$, so that $\iota_{X_{H}}\omega=dH$. Fix
		$E\in\R$; the energy level set is $\Sigma_E=\{H=E\}$. The \emph{section}
		is the hypersurface $C=\{q_{1}=0\}$.
	\end{definition}
	
	\begin{definition}[Coisotropic structure]\label{def:coiso}
		Every hypersurface is coisotropic: at each $x\in C$ the symplectic
		orthogonal $T_{x}C^{\omega}\subset T_{x}C$ is one-dimensional. For
		$C=\{q_{1}=0\}$, $\iota_{\p_{p_{1}}}\omega=-\,dq_{1}$ vanishes on $TC$, so
		$T_{x}C^{\omega}=\R\,\p_{p_{1}}$.
	\end{definition}
	
	\begin{definition}[Characteristic foliation]\label{def:foliation}
		The leaves of the line field $T C^{\omega}$ are the lines
		$\{(0,q_{2},\,\cdot\,,p_{2})\}$ parametrized by $p_{1}$. The foliation is
		denoted $\mathcal F$.
	\end{definition}
	
	\begin{definition}[Reduced space and projection]\label{def:reduced}
		The leaf space $\Cred:=C/\mathcal F\cong\R^{2}$ has coordinates
		$b=(q_{2},p_{2})$ and reduced form $\omega_{\mathrm{red}}=dq_{2}\wedge
		dp_{2}$. The projection $\pi\colon C\to\Cred$,
		$\pi(0,q_{2},p_{1},p_{2})=(q_{2},p_{2})$, is a surjective submersion and
		$\pi^{*}\omega_{\mathrm{red}}=\omega|_{C}$, because $dq_{1}|_{C}=0$. This
		is the coisotropic reduction of $C$ \cite{abraham1978}; in the present
		coordinates every assertion is a direct computation.
	\end{definition}
	
	\begin{definition}[Fibre function]\label{def:observable}
		For $b=(q_{2},p_{2})\in\Cred$,
		\[
		\hb\colon\R\to\R,\qquad \hb(p_{1})=H(0,q_{2},p_{1},p_{2}),
		\]
		is the restriction of $H|_{C}$ to the leaf over $b$. In general $H|_{C}$
		is not constant along leaves and does not descend to $\Cred$.
	\end{definition}
	
	\begin{definition}[Reconstruction correspondence]\label{def:corr}
		Let
		\[
		\Ssec:=\Sigma_E\cap C=\{(b,p_{1})\mid \hb(p_{1})=E\},\qquad
		\pi_E:=\pi|_{\Ssec}\colon\Ssec\to\Cred .
		\]
		The pair $(\Ssec,\pi_E)$ is the \emph{reconstruction correspondence}. A
		\emph{reconstruction} over $U\subseteq\Cred$ is a continuous section of
		$\pi_E$ over $U$.
	\end{definition}
	
	\begin{definition}[Admissible region]\label{def:adm}
		$\Areg:=\pi_E(\Ssec)$ is the set of $b$ for which \eqref{eq:constraint} has
		a real root.
	\end{definition}
	
	\begin{definition}[Fibrewise least energy]\label{def:marginal}
		$\Ered(b):=\inf_{p_{1}\in\R}\hb(p_{1})\in[-\infty,\infty)$.
	\end{definition}
	
	\begin{definition}[Critical locus and discriminant]\label{def:disc}
		\[
		\Crit(\pi_E):=\{\,x\in\Ssec\mid \p_{p_{1}}H(x)=0\,\},\qquad
		\DE:=\pi_E(\Crit(\pi_E)).
		\]
		Thus $b\in\DE$ if and only if $\hb-E$ has a real root that is also a
		critical point of $\hb$. The set $\Ssec$ need not be a manifold; a point
		$x\in\Ssec$ is \emph{regular} if $d(H|_{C})(x)\neq0$, and then $\Ssec$ is a
		$C^{2}$ surface near $x$ with $T_{x}\Ssec=\ker d(H|_{C})(x)$. At a regular
		point, $x\in\Crit(\pi_E)$ if and only if the characteristic direction
		$\p_{p_{1}}$ is tangent to $\Ssec$, that is, if and only if $d\pi_E(x)$
		is singular.
	\end{definition}
	
	The next proposition ties these objects to the flow. Nothing in it is
	deep; it records why the reduction of $C$ is the relevant structure.
	
	\begin{proposition}[flow tangency and the induced form]\label{prop:flow}
		Let $H\in C^{1}$.
		\begin{enumerate}[label=\textup{(\alph*)},leftmargin=2.4em]
			\item For $x\in C$, $dH(\p_{p_{1}})(x)=\omega(X_{H},\p_{p_{1}})(x)
			=dq_{1}(X_{H})(x)=\dot q_{1}$. Hence the leaf through $x$ is tangent to
			the level set of $H$ at $x$ if and only if $X_{H}(x)$ is tangent to $C$,
			and
			\[
			\Crit(\pi_E)=\{\,x\in\Ssec\mid X_{H}(x)\in T_{x}C\,\}.
			\]
			\item On the regular part of $\Ssec$,
			$\omega|_{\Ssec}=\pi_E^{*}(dq_{2}\wedge dp_{2})$, and this $2$-form is
			nondegenerate at $x$ if and only if $x\notin\Crit(\pi_E)$.
		\end{enumerate}
	\end{proposition}
	
	\begin{proof}
		(a) $\iota_{X_{H}}\omega=dH$ gives $dH(\p_{p_{1}})=\omega(X_{H},\p_{p_{1}})
		=-(\iota_{\p_{p_{1}}}\omega)(X_{H})=dq_{1}(X_{H})$, and $dq_{1}(X_{H})=\p
		H/\p p_{1}$. (b) Restrict $\pi^{*}\omega_{\mathrm{red}}=\omega|_{C}$ to
		$\Ssec$. A pulled-back area form $\pi_E^{*}\omega_{\mathrm{red}}$ on the
		two-dimensional space $T_{x}\Ssec$ is nondegenerate exactly when
		$d\pi_E(x)$ is injective, that is, when $\p_{p_{1}}\notin T_{x}\Ssec$.
	\end{proof}
	
	\begin{remark}[relation to the return map]\label{rem:return}
		Where the flow is transverse to $C$, the first-return map of $\Ssec$, when
		it is defined, preserves $\omega|_{\Ssec}$ \cite{arnold,meyer}. By
		Proposition~\ref{prop:flow}(b) it preserves $dq_{2}\wedge dp_{2}$ when
		read through $\pi_E$: away from $\Crit(\pi_E)$, $\pi_E$ is the usual
		area-preserving chart of the Poincar\'e map, and $\Crit(\pi_E)$ is where
		the chart and the transversality of the flow fail together. Part (a) is
		not tied to the coordinates: for a hypersurface $\{f=0\}$ the
		characteristic direction is spanned by $X_{f}$ and
		$dH(X_{f})=-df(X_{H})$.
	\end{remark}
	
	\begin{remark}[locality over the base]\label{rem:local-base}
		All constructions and results below are local over the base and global
		along the leaves. If $H$ is $C^{2}$ only on $\pi^{-1}(B)$ for an open set
		$B\subseteq\Cred$, they hold with $\Cred$ replaced by $B$. This is used
		in Section~\ref{ssec:ss}.
	\end{remark}
	
	\section{Structural hypotheses}\label{sec:hyp}
	
	\subsection{The hypotheses}\label{ssec:hyp}
	
	Three hypotheses concern the fibre functions, one the function $H|_{C}$ at a
	point, and one the base.
	\begin{itemize}[leftmargin=2.6em]
		\item[(C)] \emph{Coercivity.} For every compact $K\subseteq\Cred$ and
		every $M\in\R$ there is $R>0$ such that $\hb(p_{1})>M$ whenever $b\in K$
		and $|p_{1}|>R$.
		\item[(U)] \emph{Unimodality.} For every $b$ there is a point $m(b)$ with
		$\hb'<0$ on $(-\infty,m(b))$ and $\hb'>0$ on $(m(b),\infty)$.
		\item[(N)] \emph{Fibrewise nondegeneracy} at $x_{0}\in C$ with
		$\p_{p_{1}}H(x_{0})=0$: $\p^{2}_{p_{1}p_{1}}H(x_{0})\neq0$. We say that (N)
		holds, without qualification, if it holds at every such point.
		\item[(T)] \emph{Transversality} at $x_{0}\in\Crit(\pi_E)$:
		$(\p_{q_{2}}H,\p_{p_{2}}H)(x_{0})\neq0$. Since $\p_{p_{1}}H(x_{0})=0$, this
		says that $x_{0}$ is a regular point of $\Ssec$.
		\item[(R)] \emph{Regularity of the level on the base.} $\Ered$ is
		differentiable at every point of $\{\Ered=E\}$ and $d\Ered\neq0$ there.
	\end{itemize}
	Under (U) the point $m(b)$ is unique, $\hb'(m(b))=0$, and $m(b)$ is the
	unique critical point and the unique global minimizer of $\hb$. Hypotheses
	(C), (U), (N) are fibrewise; (R) is a condition on a function of $b$. Their
	relation to (T) is the content of Proposition~\ref{prop:reg}.
	
	\subsection{The least-energy function and what it forgets}\label{ssec:loss}
	
	\begin{lemma}[consequences of coercivity]\label{lem:coercive}
		Assume \textup{(C)}.
		\begin{enumerate}[label=\textup{(\alph*)},leftmargin=2.4em]
			\item The infimum defining $\Ered(b)$ is attained, $\Ered$ is
			continuous, and $\Areg=\{\Ered\le E\}$; in particular $\Areg$ is closed.
			\item If in addition \textup{(U)} holds, then $m$ is continuous,
			$\Ered(b)=\hb(m(b))$, and $\Ered\in C^{1}(\Cred)$ with
			\begin{equation}\label{eq:envelope1}
				d\Ered(b)=d_{b}H\big|_{p_{1}=m(b)} .
			\end{equation}
		\end{enumerate}
	\end{lemma}
	
	\begin{proof}
		(a) Fix a compact $K$, put $M=\max_{b\in K}\hb(0)$ and take $R$ from (C).
		For $b\in K$, $\Ered(b)=\min_{|p_{1}|\le R}\hb(p_{1})$, a minimum of a
		continuous function over a fixed compact set, hence attained and
		continuous in $b$. If $\Ered(b)\le E$, the continuous function $\hb$ takes
		the value $\Ered(b)\le E$ and tends to $+\infty$, so it takes the value $E$;
		the converse is trivial. (b) Let $b_{n}\to b$. The points $m(b_{n})$ lie in
		$[-R,R]$, and any limit point $p^{*}$ satisfies
		$\hb(p^{*})=\lim\Ered(b_{n})=\Ered(b)$, so $p^{*}=m(b)$ by uniqueness of the
		minimizer. For the derivative, minimality gives, for $b'$ near $b$,
		\[
		H(b',m(b'))-H(b,m(b'))\;\le\;\Ered(b')-\Ered(b)\;\le\;H(b',m(b))-H(b,m(b)),
		\]
		where we write $H(b,p_{1})$ for $H(0,q_{2},p_{1},p_{2})$. Both bounds equal
		$d_{b}H(b,m(b))(b'-b)+o(|b'-b|)$, the left one by the mean value theorem
		and the continuity of $d_{b}H$ and of $m$. This gives
		\eqref{eq:envelope1}, whose right-hand side is continuous.
	\end{proof}
	
	The infimum is canonical in the following elementary sense.
	
	\begin{proposition}[universal property]\label{prop:universal}
		Order the extended-real functions on $C$ and on $\Cred$ pointwise and let
		$\pi^{*}f=f\circ\pi$. The operator $\mathrm R(g)(b)=\inf_{\pi^{-1}(b)}g$ is
		the unique map with
		\begin{equation}\label{eq:adjunction}
			\pi^{*}f\le g\iff f\le\mathrm R(g)\qquad(\forall f,g).
		\end{equation}
		In particular $\Ered=\mathrm R(H|_{C})$ is the largest function on $\Cred$
		whose pullback is dominated by $H|_{C}$.
	\end{proposition}
	
	\begin{proof}
		$\pi^{*}f\le g$ means $f(\pi(c))\le g(c)$ for all $c$, that is,
		$f(b)\le\inf_{\pi^{-1}(b)}g$ for all $b$. Any $\mathrm R'$ satisfying
		\eqref{eq:adjunction} obeys $\mathrm R'(g)=\max\{f:\pi^{*}f\le g\}=\mathrm
		R(g)$.
	\end{proof}
	
	\begin{proposition}[the least energy does not determine the
		correspondence]\label{prop:lossy}
		The assignment $H|_{C}\mapsto\Ered$ is not injective, even among smooth
		observables satisfying \textup{(C)}, \textup{(U)} and \textup{(N)}: for any
		$W\in C^{\infty}(\Cred)$ the fibre functions $\tfrac12p_{1}^{2}+W(b)$ and
		$2p_{1}^{2}+W(b)$ have the same $\Ered=W$, hence the same $\Areg$ for every
		$E$, and different sets $\Ssec$, namely $p_{1}=\pm\sqrt{2(E-W)}$ and
		$p_{1}=\pm\sqrt{(E-W)/2}$.
	\end{proposition}
	
	By Lemma~\ref{lem:coercive}(a), $\Ered$ determines the admissible region. By
	Proposition~\ref{prop:lossy} it does not determine $\Ssec$: it does not
	contain the fibre curvature $\hb''$, which fixes the rate
	$\sqrt{2(E-\Ered)/\hb''}$ at which the two lifts separate
	(Theorem~\ref{thm:fold}, Corollary~\ref{cor:fingerprint}). This is the reason
	for working with $(\Ssec,\pi_E)$ rather than with $\Ered$.
	
	\subsection{Momentum convexity}\label{ssec:convex}
	
	\begin{definition}\label{def:convex}
		$H$ has a \emph{positive-definite momentum Hessian}, written
		$D^{2}_{pp}H\succ0$, if the matrix $(\p^{2}H/\p p_{i}\p p_{j})$ is positive
		definite at every point. For a natural mechanical Hamiltonian
		$H=\tfrac12p^{\mathsf T}M(q)^{-1}p+V(q)$ with $M(q)\succ0$
		\cite{arnold,meyer} one has $D^{2}_{pp}H=M(q)^{-1}\succ0$, and \textup{(C)}
		holds because $\hb$ is a quadratic polynomial in $p_{1}$ whose leading
		coefficient is positive and continuous in $b$.
	\end{definition}
	
	\begin{proposition}[what momentum convexity provides]\label{prop:convex}
		Suppose that \textup{(C)} holds and that $D^{2}_{pp}H\succ0$. Then each
		fibre function $\hb$ is coercive and strictly convex. Consequently it has
		a unique critical point $m(b)$, which is its unique global minimizer, and
		\textup{(U)} and \textup{(N)} hold, with $\hb''(m(b))>0$.
	\end{proposition}
	
	\begin{proof}
		The $(1,1)$ entry of a positive-definite matrix is positive, so $\hb''>0$
		and $\hb'$ is strictly increasing: this is strict convexity. By (C) the
		continuous function $\hb$ attains its minimum at some $m(b)$, where
		$\hb'(m(b))=0$. A strictly increasing function has at most one zero, and
		it is negative before it and positive after it: this is (U). Finally
		$\hb''(m(b))>0$ is (N).
	\end{proof}
	
	Each ingredient is needed, and the two notions of convexity differ.
	
	\begin{counterexample}[positive Hessian without coercivity]\label{cex:coer}
		For $H=e^{p_{1}}+\tfrac12p_{2}^{2}+V(q)$ one has
		$D^{2}_{pp}H=\operatorname{diag}(e^{p_{1}},1)\succ0$, but
		$\hb(p_{1})=e^{p_{1}}+W(b)$, $W=\tfrac12p_{2}^{2}+V(0,q_{2})$, has no
		critical point. \textup{(C)} and \textup{(U)} both fail, $\Ered=W$ is not
		attained, $\Areg=\{W<E\}$ is open, and every fibre over $\Areg$ has one
		point. Positivity of the momentum Hessian alone gives neither
		\textup{(U)} nor the count of Proposition~\ref{prop:mult}.
	\end{counterexample}
	
	\begin{counterexample}[strict convexity without nondegeneracy]\label{cex:nd}
		For $H=p_{1}^{4}+\tfrac12p_{2}^{2}+V(q)$ every $\hb=p_{1}^{4}+W(b)$ is
		coercive and strictly convex, so \textup{(C)} and \textup{(U)} hold, but
		$\hb''(0)=0$ and \textup{(N)} fails. The two roots are $\pm(E-W)^{1/4}$,
		and at a point of $\Crit(\pi_E)$ with $dW\neq0$ the projection is
		$(u,v)\mapsto(u,v^{4})$ in suitable coordinates, which is not a fold.
		Strict convexity of $\hb$ in the ordinary sense, together with
		\textup{(C)}, gives \textup{(U)} by the proof above but does not give
		\textup{(N)}.
	\end{counterexample}
	
	In short: $D^{2}_{pp}H\succ0$ with (C) implies (U) and (N); without (C) it
	implies neither (U) nor the fibre count. Its effect on the boundary of the
	admissible region is stronger than on the fibres and is the subject of
	Corollary~\ref{cor:convexbdry}.
	
	\subsection{Terminology: fold-branched double covering}\label{ssec:term}
	
	\begin{definition}\label{def:branched}
		We say that $\pi_E$ is a \emph{fold-branched double covering} of $\Areg$
		if
		\begin{enumerate}[label=\textup{(\roman*)},leftmargin=2.6em]
			\item $\DE=\p\Areg$;
			\item over $\intr\Areg$ the map $\pi_E$ is a two-sheeted covering;
			\item every point of $\Crit(\pi_E)$ is a regular point of $\Ssec$ at
			which $\pi_E$ is a Whitney fold, that is, $(u,v)\mapsto(u,v^{2})$ in
			suitable local coordinates \cite[Ch.~III,\S4]{golubitsky1973}.
		\end{enumerate}
	\end{definition}
	
	This is the meaning of ``branched covering'' in the title and wherever the
	expression is used below. It is not a branched covering in the sense of
	low-dimensional topology, where the branch set has codimension two and the
	map is open: here the two sheets are joined along the curve $\DE$, the image
	$\Areg$ is a surface with boundary, and $\Ssec$ is obtained from two copies of
	$\Areg$ glued along $\p\Areg$. In local coordinates it is the real picture
	$\{v^{2}=s\}\to\{s\ge0\}$ of a double cover branched along $\{s=0\}$.
	Theorem~\ref{thm:structure} gives hypotheses under which $\pi_E$ has this
	structure.
	
	\section{Reconstruction as a covering with fold branching}
	\label{sec:structure}
	
	\begin{lemma}[local structure away from the critical locus]\label{lem:local}
		Let $H\in C^{1}$ and $x_{0}\in\Ssec$ with $\p_{p_{1}}H(x_{0})\neq0$. Near
		$x_{0}$ the set $\Ssec$ is the graph of a $C^{1}$ function
		$p_{1}=\varphi(b)$, and $\pi_E$ is a local $C^{1}$ diffeomorphism at
		$x_{0}$.
	\end{lemma}
	
	\begin{proof}
		Apply the implicit function theorem to $\Phi(b,p_{1})=\hb(p_{1})-E$ at
		$x_{0}$. In the chart $b\mapsto(b,\varphi(b))$ the map $\pi_E$ is the
		identity.
	\end{proof}
	
	\begin{theorem}[properness and covering]\label{thm:cover}
		Assume $H\in C^{1}$ and \textup{(C)}.
		\begin{enumerate}[label=\textup{(\alph*)},leftmargin=2.4em]
			\item $\pi_E$ is proper; in particular it is a closed map and $\DE$ is
			closed.
			\item Over the open set $\Cred\setminus\DE$ the restriction
			\[
			\pi_E\colon\pi_E^{-1}(\Cred\setminus\DE)\longrightarrow\Cred\setminus\DE
			\]
			is a proper local diffeomorphism with finitely many sheets: every
			$b_{0}\notin\DE$ has a neighbourhood $W$ such that $\pi_E^{-1}(W)$ is a
			disjoint union of $n\ge0$ open sets, each mapped diffeomorphically onto
			$W$. The number $n$ of sheets is the number of roots of
			\eqref{eq:constraint}; it is constant on each connected component of
			$\Cred\setminus\DE$ and may be zero. In particular $\pi_E$ is a
			finite-sheeted covering of its image $\Areg\setminus\DE$, which is a
			union of such components.
		\end{enumerate}
	\end{theorem}
	
	\begin{proof}
		(a) Let $K\subseteq\Cred$ be compact. By (C) there is $R$ with
		$\hb(p_{1})>E$ for $b\in K$, $|p_{1}|>R$, so
		$\pi_E^{-1}(K)\subseteq K\times[-R,R]$; it is closed because $\Ssec$ is.
		A proper continuous map into $\R^{2}$ is closed, and $\Crit(\pi_E)$ is
		closed in $\Ssec$, so $\DE=\pi_E(\Crit(\pi_E))$ is closed.
		
		(b) Let $b_{0}\notin\DE$. The fibre $\pi_E^{-1}(b_{0})$ is compact by (a)
		and consists of points where $\pi_E$ is a local diffeomorphism
		(Lemma~\ref{lem:local}), so it is finite, say $\{x_{1},\dots,x_{n}\}$,
		$n\ge0$. Choose disjoint open sets $V_{i}\ni x_{i}$ in $\Ssec$ mapped
		diffeomorphically onto open sets. The set
		$\pi_E(\Ssec\setminus\bigcup V_{i})$ is closed and does not contain
		$b_{0}$, so
		\[
		W:=\Bigl(\bigcap_{i}\pi_E(V_{i})\Bigr)\setminus
		\Bigl(\DE\cup\pi_E\bigl(\Ssec\setminus\textstyle\bigcup V_{i}\bigr)\Bigr)
		\]
		(with $W:=\Cred\setminus(\DE\cup\pi_E(\Ssec))$ if $n=0$) is an open
		neighbourhood of $b_{0}$ with $\pi_E^{-1}(W)\subseteq\bigcup V_{i}$, and
		$\pi_E^{-1}(W)\cap V_{i}\to W$ is a diffeomorphism for each $i$.
	\end{proof}
	
	\begin{proposition}[fibrewise multiplicity]\label{prop:mult}
		Assume \textup{(C)} and \textup{(U)}. For every $b$,
		\[
		\#\pi_E^{-1}(b)=
		\begin{cases}
			2,& \Ered(b)<E,\\
			1,& \Ered(b)=E,\\
			0,& \Ered(b)>E,
		\end{cases}
		\]
		and $\DE=\{\Ered=E\}$.
	\end{proposition}
	
	\begin{proof}
		Fix $b$ and write $m=m(b)$, so that $\Ered(b)=\hb(m)$. If $\Ered(b)<E$, then
		on $(-\infty,m)$ the function $\hb$ decreases strictly from $+\infty$ to
		$\Ered(b)<E$ and takes the value $E$ exactly once; likewise on
		$(m,\infty)$. If $\Ered(b)=E$ the only root is $m$; if $\Ered(b)>E$ there is
		none. The only critical point of $\hb$ is $m$, and it lies on $\Ssec$ if
		and only if $\Ered(b)=E$.
	\end{proof}
	
	Only the monotonicity supplied by (U) and the coercivity are used; neither
	convexity nor (N) enters.
	
	\begin{corollary}[the two sheets; the forward-crossing sheet]
		\label{cor:sheets}
		Assume $H\in C^{1}$, \textup{(C)} and \textup{(U)}, and denote by
		$p^{-}(b)<m(b)<p^{+}(b)$ the two roots over $b\in\{\Ered<E\}$.
		\begin{enumerate}[label=\textup{(\alph*)},leftmargin=2.4em]
			\item $p^{\pm}\in C^{1}(\{\Ered<E\})$, and
			$\Ssec\setminus\Crit(\pi_E)$ is the disjoint union of the graphs of
			$p^{-}$ and $p^{+}$: the covering of $\{\Ered<E\}$ is trivial.
			\item $p^{\pm}$ extend continuously to $\{\Ered\le E\}$, with
			$p^{\pm}=m$ on $\DE$.
			\item The reconstructions over a connected open set
			$U\subseteq\{\Ered<E\}$ are exactly $p^{-}|_{U}$ and $p^{+}|_{U}$.
			\item The forward-crossing part of $\Ssec$ is the sheet $p^{+}$:
			\[
			\{x\in\Ssec\mid \dot q_{1}>0\}=\operatorname{graph}(p^{+}),
			\]
			and $\pi_E$ maps it $C^{1}$-diffeomorphically onto $\{\Ered<E\}$,
			carrying $\omega|_{\Ssec}$ to $dq_{2}\wedge dp_{2}$.
		\end{enumerate}
	\end{corollary}
	
	\begin{proof}
		(a) At $(b,p^{\pm}(b))$ one has $\hb'\neq0$ by (U), so
		Lemma~\ref{lem:local} applies; the two graphs are distinguished by the
		sign of $\hb'$. (b) Let $b_{n}\to b_{0}\in\DE$ with $\Ered(b_{n})<E$. By
		properness $p^{\pm}(b_{n})$ is bounded, and every limit point is a root of
		$h_{b_{0}}=E$, hence equals $m(b_{0})$ by Proposition~\ref{prop:mult}. (c) A
		continuous section $s$ over $U$ takes values in $\{p^{-},p^{+}\}$ and the
		sign of $\hb'(s(b))$ is continuous and nonvanishing, hence constant on $U$.
		(d) By Proposition~\ref{prop:flow}(a), $\dot q_{1}=\hb'(p_{1})$, which is
		positive exactly for $p_{1}>m(b)$; the root in that range is $p^{+}(b)$
		and exists exactly when $\Ered(b)<E$. The statement on forms is
		Proposition~\ref{prop:flow}(b).
	\end{proof}
	
	Thus, under (C) and (U), the usual Poincar\'e section ``$q_{1}=0$,
	$\dot q_{1}>0$ at energy $E$'' is the open set $\{\Ered<E\}$ of the reduced
	plane, the point over $b$ being $(b,p^{+}(b))$. The choice between $p^{+}$ and
	$p^{-}$ is a convention about the direction of crossing.
	
	\begin{lemma}[critical branch and envelope identity]\label{lem:branch}
		Let $H\in C^{k}$ near $x_{0}=(b_{0},p_{1}^{0})\in C$, $k\ge2$, with
		$\p_{p_{1}}H(x_{0})=0$ and \textup{(N)} at $x_{0}$. Then there is a unique
		$C^{k-1}$ function $p_{1}^{c}(b)$ near $b_{0}$ with
		$p_{1}^{c}(b_{0})=p_{1}^{0}$ and $\p_{p_{1}}H(b,p_{1}^{c}(b))=0$, and the
		critical value $\Ered^{c}(b):=H(b,p_{1}^{c}(b))$ is of class $C^{k}$ and
		satisfies
		\begin{equation}\label{eq:envelope}
			d\Ered^{c}(b)=d_{b}H\big|_{p_{1}=p_{1}^{c}(b)} .
		\end{equation}
		Under \textup{(U)}, $p_{1}^{c}=m$ and $\Ered^{c}=\Ered$ near $b_{0}$.
	\end{lemma}
	
	\begin{proof}
		The implicit function theorem applied to $\p_{p_{1}}H=0$ gives
		$p_{1}^{c}\in C^{k-1}$. Differentiating $\Ered^{c}$ and using
		$\p_{p_{1}}H=0$ along the branch gives \eqref{eq:envelope}. Its right-hand
		side is a composition of $C^{k-1}$ maps, so $\Ered^{c}\in C^{k}$. Under
		(U) the only critical point of $\hb$ is $m(b)$.
	\end{proof}
	
	\begin{theorem}[fold]\label{thm:fold}
		Let $H$ be smooth near $x_{0}\in\Crit(\pi_E)$, and assume \textup{(N)} and
		\textup{(T)} at $x_{0}$. Then $\Ssec$ is a smooth surface near $x_{0}$ and
		there are local coordinates $(u,v)$ on $\Ssec$ and $(u,s)$ on $\Cred$,
		centred at $x_{0}$ and $b_{0}=\pi(x_{0})$, in which
		\[
		\pi_E\colon(u,v)\longmapsto\bigl(u,\,-\tfrac12a\,v^{2}\bigr),\qquad
		a=\p^{2}_{p_{1}p_{1}}H(x_{0}),\quad s=\Ered^{c}(b)-E .
		\]
		After rescaling $s$ this is the Whitney fold $(u,v)\mapsto(u,v^{2})$
		\cite[Ch.~III,\S4]{golubitsky1973}. The two preimages lie over the side
		$\{a\,s<0\}$ of the curve $\{\Ered^{c}=E\}$; for $a>0$ this is
		$\{\Ered^{c}<E\}$.
	\end{theorem}
	
	\begin{proof}
		By (T), $d(H|_{C})(x_{0})\neq0$, so $\Ssec$ is a smooth surface near
		$x_{0}$. Let $p_{1}^{c}$ and $\Ered^{c}$ be as in Lemma~\ref{lem:branch}
		and $\Phi(b,p_{1})=H(b,p_{1})-E$. Taylor's formula with integral remainder
		about $p_{1}^{c}(b)$, where $\p_{p_{1}}\Phi$ vanishes, gives
		\[
		\Phi(b,p_{1})=\bigl(\Ered^{c}(b)-E\bigr)+(p_{1}-p_{1}^{c}(b))^{2}\,g(b,p_{1}),
		\quad
		g=\int_{0}^{1}(1-t)\,\p^{2}_{p_{1}p_{1}}H\bigl(b,p_{1}^{c}(b)
		+t(p_{1}-p_{1}^{c}(b))\bigr)\,dt ,
		\]
		with $g$ smooth and $g(x_{0})=a/2\neq0$. Put
		$v=(p_{1}-p_{1}^{c}(b))\sqrt{2g(b,p_{1})/a}$. Then $v$ is smooth near
		$x_{0}$, $\p v/\p p_{1}(x_{0})=1$, so $(b,v)$ are coordinates on $C$ near
		$x_{0}$, and
		\[
		\Phi=\tfrac12a\,v^{2}+s,\qquad s=\Ered^{c}(b)-E .
		\]
		(This is the Morse lemma in the variable $p_{1}$ with parameter $b$.) By
		\eqref{eq:envelope} and (T), $ds(b_{0})=d_{b}H(x_{0})\neq0$, so $s$ is part
		of a coordinate system $(u,s)$ on $\Cred$ near $b_{0}$. In the coordinates
		$(u,s,v)$ on $C$, $\Ssec=\{s=-\tfrac12a\,v^{2}\}$ is a graph over $(u,v)$,
		and $\pi_E$ is $(u,v)\mapsto(u,-\tfrac12a\,v^{2})$.
	\end{proof}
	
	\begin{remark}[finite differentiability; the branches at the fold]
		\label{rem:taylor}
		If $H\in C^{k}$ with $k\ge3$ the same proof gives the normal form in
		$C^{k-2}$ coordinates. In the original variables the theorem says that,
		for $a>0$ and $b$ near $b_{0}$ with $\Ered^{c}(b)<E$,
		\[
		p^{\pm}(b)-p_{1}^{c}(b)=\pm\sqrt{\frac{2\,(E-\Ered^{c}(b))}{a}}\,
		\bigl(1+o(1)\bigr).
		\]
		The two sheets extend continuously up to the fold curve and meet there
		(Corollary~\ref{cor:sheets}(b)), but neither is differentiable at it, and
		no reconstruction exists on the other side. Passing through the fold
		curve along $\Ssec$ leads from one sheet to the other.
	\end{remark}
	
	\begin{proposition}[regularity of the level, transversality, and
		\textup{(R)}]\label{prop:reg}
		Let $H\in C^{1}$.
		\begin{enumerate}[label=\textup{(\alph*)},leftmargin=2.4em]
			\item Every point of $\Ssec\setminus\Crit(\pi_E)$ is a regular point of
			$\Ssec$. Hence $E$ is a regular value of $H|_{C}$ if and only if
			\textup{(T)} holds at every point of $\Crit(\pi_E)$.
			\item Assume \textup{(C)} and \textup{(U)}. The critical points of
			$H|_{C}$ are the points $(b,m(b))$ with $d\Ered(b)=0$, and
			$H|_{C}(b,m(b))=\Ered(b)$. Consequently the following are equivalent:
			\begin{enumerate}[label=\textup{(\roman*)}]
				\item \textup{(R)} holds at $E$;
				\item $E$ is a regular value of $H|_{C}$;
				\item \textup{(T)} holds at every point of $\Crit(\pi_E)$.
			\end{enumerate}
			In that case $\Ssec$ is a closed embedded $C^{1}$ surface in $C$.
		\end{enumerate}
	\end{proposition}
	
	\begin{proof}
		(a) At such a point $\p_{p_{1}}H\neq0$, so $d(H|_{C})\neq0$; at a point of
		$\Crit(\pi_E)$, $d(H|_{C})\neq0$ is (T). (b) $d(H|_{C})=(d_{b}H,\p_{p_{1}}H)$
		vanishes at $(b,p_{1})$ if and only if $p_{1}=m(b)$, by (U), and
		$d_{b}H(b,m(b))=0$, which is $d\Ered(b)=0$ by \eqref{eq:envelope1}. So
		critical points and critical values of $H|_{C}$ and of $\Ered$ correspond,
		which is (i)$\Leftrightarrow$(ii); (ii)$\Leftrightarrow$(iii) is (a).
	\end{proof}
	
	Part of this is already visible in the first line of the proof of
	Theorem~\ref{thm:fold}. We state it separately because it locates (R): read on
	the total space of the section, it is the familiar requirement that the energy
	value be regular for $H|_{C}$, so that $\Ssec$ is a surface. No fibrewise
	hypothesis implies it (Section~\ref{sec:interior}), but it is not a condition
	of a new kind.
	
	The reconstruction over $b$ is the intersection of the leaf with $\Sigma_E$:
	one looks for the level $E$ on the graph of $\hb$. Under (C) and (U) the graph
	is a single well and the three cases of Proposition~\ref{prop:mult} are those
	of Figure~\ref{fig:regimes}.
	
	\begin{figure}[htbp]
		\centering
		\includegraphics[width=\linewidth]{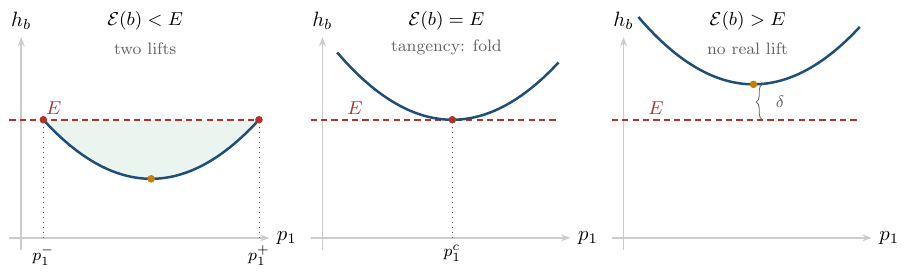}
		\caption{The three regimes of the reconstruction over a base point $b$
			under \textup{(C)} and \textup{(U)}. The solid curve is the fibre
			function $\hb$, the dashed line the energy $E$. Left: $\Ered(b)<E$, two
			lifts $p^{\pm}$. Centre: $\Ered(b)=E$, the lifts coincide at the fibre
			minimum and $b\in\DE$; this is a fold when \textup{(N)} and \textup{(T)}
			hold. Right: $\Ered(b)>E$, no lift; the gap is the energy deficit
			$\delta(b)$ of Definition~\ref{def:deficit}.}
		\label{fig:regimes}
	\end{figure}
	
	\section{Discriminant versus boundary of the admissible region}
	\label{sec:boundary}
	
	The discriminant $\DE$ is defined by a double-root condition; the boundary
	$\p\Areg$ is a topological notion. They are compared as follows.
	
	\begin{theorem}[discriminant and boundary]\label{thm:boundary}
		Assume $H\in C^{1}$ and \textup{(C)}.
		\begin{enumerate}[label=\textup{(\alph*)},leftmargin=2.4em]
			\item $\p\Areg\subseteq\DE$.
			\item Assume also \textup{(U)}, and let
			\[
			\ME:=\{\,b\in\Cred\mid \Ered(b)=E \text{ and } \Ered\le E
			\text{ on a neighbourhood of } b\,\}
			\]
			be the set of local maxima of $\Ered$ on the level $E$. Then
			\[
			\intr\Areg=\{\Ered<E\}\cup\ME,\qquad
			\p\Areg=\{\Ered=E\}\setminus\ME=\DE\setminus\ME .
			\]
			In particular
			\[
			\DE=\p\Areg\iff\ME=\emptyset ,
			\]
			and in that case $\intr\Areg=\{\Ered<E\}$.
		\end{enumerate}
	\end{theorem}
	
	\begin{proof}
		(a) $\Areg$ is closed (Lemma~\ref{lem:coercive}), so a point
		$b\in\p\Areg$ has a lift $x$. If no lift of $b$ were in $\Crit(\pi_E)$,
		Lemma~\ref{lem:local} would give a local section of $\pi_E$ through $x$
		and $b$ would be interior to $\Areg$. (b) By Lemma~\ref{lem:coercive} and
		Proposition~\ref{prop:mult}, $\Areg=\{\Ered\le E\}$ and $\DE=\{\Ered=E\}$,
		with $\Ered$ continuous. A point with $\Ered(b)<E$ is interior to $\Areg$
		by continuity. A point with $\Ered(b)=E$ is interior to $\Areg$ if and only
		if $\Ered\le E$ on a neighbourhood of it, that is, if and only if
		$b\in\ME$. The remaining assertions follow.
	\end{proof}
	
	Part (b) uses only the continuity of $\Ered$ and the two identities
	$\Areg=\{\Ered\le E\}$, $\DE=\{\Ered=E\}$. The points of the discriminant that
	fail to be boundary points are exactly the local maxima of the least-energy
	function at height $E$; we call them \emph{interior discriminant points}.
	Over such a point the fibre has one element, while nearby fibres have two.
	(Without (U), $\DE$ may in addition contain images of fibre critical points
	that are not minima; we do not pursue this.)
	
	\begin{corollary}[regularity is sufficient, not necessary]\label{cor:R}
		Assume $H\in C^{1}$, \textup{(C)} and \textup{(U)}. If \textup{(R)} holds
		at $E$, then $\ME=\emptyset$ and $\DE=\p\Areg$. The converse is false.
	\end{corollary}
	
	\begin{proof}
		A local maximum of the $C^{1}$ function $\Ered$ is a critical point, so
		(R) excludes $\ME\neq\emptyset$. For the converse take
		$H=\tfrac12p_{1}^{2}+q_{2}^{2}-p_{2}^{2}$ and $E=0$: here
		$\Ered=q_{2}^{2}-p_{2}^{2}$ has a critical point on the level $0$, so (R)
		fails, but the origin is a saddle, $\ME=\emptyset$, and
		$\DE=\{|q_{2}|=|p_{2}|\}=\p\{|q_{2}|\le|p_{2}|\}=\p\Areg$.
	\end{proof}
	
	\begin{corollary}[momentum convexity removes the obstruction]
		\label{cor:convexbdry}
		Assume \textup{(C)} and $D^{2}_{pp}H\succ0$. Then, for each $q_{2}$, the
		function $p_{2}\mapsto\Ered(q_{2},p_{2})$ is strictly convex. Consequently
		$\Ered$ has no local maximum, $\ME=\emptyset$ for every $E$, and
		\[
		\DE=\p\Areg\qquad\text{for every } E\in\R,
		\]
		whether or not \textup{(R)} holds. This applies to every natural
		mechanical Hamiltonian $H=\tfrac12p^{\mathsf T}M(q)^{-1}p+V(q)$.
	\end{corollary}
	
	\begin{proof}
		By Proposition~\ref{prop:convex}, (U) holds, so
		Theorem~\ref{thm:boundary}(b) applies. Fix $q_{2}$. The function
		$(p_{1},p_{2})\mapsto H(0,q_{2},p_{1},p_{2})$ has positive-definite
		Hessian and is therefore strictly convex. Let $p_{2}\neq p_{2}'$,
		$0<t<1$, and let $m,m'$ be the fibre minimizers over $(q_{2},p_{2})$ and
		$(q_{2},p_{2}')$. Then
		\begin{align*}
			\Ered\bigl(q_{2},tp_{2}+(1-t)p_{2}'\bigr)
			&\le H\bigl(0,q_{2},tm+(1-t)m',tp_{2}+(1-t)p_{2}'\bigr)\\
			&< t\,\Ered(q_{2},p_{2})+(1-t)\,\Ered(q_{2},p_{2}') .
		\end{align*}
		A strictly convex function of one variable has no local maximum, so no
		point is a local maximum of $\Ered$. For a natural mechanical Hamiltonian
		(C) holds by Definition~\ref{def:convex}.
	\end{proof}
	
	The general theory thus isolates the exact obstruction to $\DE=\p\Areg$,
	namely a local maximum of $\Ered$ at the prescribed energy, and
	positive-definite momentum convexity with coercive fibres removes that
	obstruction automatically. What convexity does not remove is the possible
	failure of (R): at a critical value of $\Ered$ the identity $\DE=\p\Areg$
	persists, but $\Ssec$ is singular at the corresponding points and $\pi_E$ is
	not a fold there (Remark~\ref{rem:morse}, Section~\ref{ssec:hh}).
	
	\begin{theorem}[fold-branched double covering]\label{thm:structure}
		Let $H$ be smooth and assume \textup{(C)}, \textup{(U)} and \textup{(N)}.
		If \textup{(R)} holds at $E$, then $\pi_E$ is a fold-branched double
		covering of $\Areg$ in the sense of Definition~\ref{def:branched}: $\Ssec$
		is a smooth closed surface, $\DE=\p\Areg=\{\Ered=E\}$ is a smooth curve,
		$\pi_E$ is a trivial two-sheeted covering over $\intr\Areg=\{\Ered<E\}$
		with sheets $p^{\pm}$, and $\pi_E$ is a Whitney fold along
		$\Crit(\pi_E)$. Conversely, under \textup{(C)} and \textup{(U)}, if $\pi_E$
		is a fold-branched double covering then \textup{(R)} holds at $E$.
	\end{theorem}
	
	\begin{proof}
		By Proposition~\ref{prop:reg}, (R) gives (T) at every point of
		$\Crit(\pi_E)$ and makes $\Ssec$ a regular level set; with (N),
		Theorem~\ref{thm:fold} gives the fold. Corollary~\ref{cor:R} gives
		$\DE=\p\Areg$ and $\intr\Areg=\{\Ered<E\}$, and
		Corollary~\ref{cor:sheets} the two sheets. The curve $\{\Ered=E\}$ is
		smooth because $\Ered$ is smooth (Lemma~\ref{lem:branch}) with $d\Ered\neq0$
		on it. Conversely, condition (iii) of Definition~\ref{def:branched}
		contains (T) at every point of $\Crit(\pi_E)$, which is (R) by
		Proposition~\ref{prop:reg}.
	\end{proof}
	
	\begin{remark}[genericity in $E$]\label{rem:sard}
		Let $H\in C^{2}$ satisfy (C), (U) and (N). Then $\Ered\in C^{2}(\Cred)$ by
		Lemma~\ref{lem:branch}, and Sard's theorem for $C^{2}$ functions on the
		plane \cite{milnor1965} shows that (R) holds for almost every $E$. (For
		$C^{1}$ functions of two variables the conclusion of Sard's theorem can
		fail, which is why (N) is used here.) The critical values of $\Ered$ are
		the only energies at which the topology of $\Areg=\{\Ered\le E\}$ can
		change; whether it does change at a given critical value requires
		separate, Morse-type information.
	\end{remark}
	
	\begin{remark}[a discriminant of $\pi_E$]\label{rem:disc-not-caustic}
		The singular behaviour belongs to the map $\pi_E$ and not to the function
		$\Ered$: under (C), (U) and (N), $\Ered$ is as smooth as $H$. In
		particular $\DE$ is not a caustic of the Lagrangian graph of $d\Ered$,
		which projects diffeomorphically onto $\Cred$. It is a discriminant in the
		classical sense: the set of parameters $b$ for which $\hb-E$ has a double
		real root.
	\end{remark}
	
	\begin{remark}[which boundary]\label{rem:hill}
		For a natural mechanical system, $\p\Areg$ should not be confused with the
		zero-velocity curve bounding Hill's region in configuration space. On
		$\p\Areg$ only $p_{1}$ vanishes: by Proposition~\ref{prop:flow},
		$\p\Areg=\DE$ is the projection of the set where the flow is tangent to
		the section.
	\end{remark}
	
	\section{Interior discriminant points}\label{sec:interior}
	
	\begin{counterexample}[nondegenerate coercive fibres, interior
		discriminant point]\label{cex:hat}
		Let
		\[
		H=\tfrac12p_{1}^{2}+W(q_{2},p_{2}),\qquad
		W=(q_{2}^{2}+p_{2}^{2}-1)^{2},
		\]
		so that $\hb(p_{1})=\tfrac12p_{1}^{2}+W(b)$. Every fibre function is a
		nondegenerate coercive parabola: \textup{(C)}, \textup{(U)} and
		\textup{(N)} hold, $m\equiv0$ and $\Ered=W$. Take $E=1$ and write
		$r^{2}=q_{2}^{2}+p_{2}^{2}$. Then
		\[
		\Areg=\{W\le1\}=\{r\le\sqrt2\},\qquad \p\Areg=\{r=\sqrt2\},\qquad
		\DE=\{W=1\}=\{r=\sqrt2\}\cup\{0\}.
		\]
		The origin is a local maximum of $\Ered$ with $\Ered(0)=1=E$, so
		$\ME=\{0\}$: it is an interior discriminant point, and
		$\DE\supsetneq\p\Areg$. The fibre over the origin is the single point
		$p_{1}=0$, and the fibres over $0<r<\sqrt2$ have two points. Near
		$x_{0}=(0,0,0)$,
		\[
		\Ssec=\{\,p_{1}^{2}=2r^{2}(2-r^{2})\,\}
		\]
		is a cone with vertex $x_{0}$: \textup{(T)} fails there, $1$ is a critical
		value of $H|_{C}$ and of $\Ered$, \textup{(R)} fails at $E=1$, and $\pi_E$
		is not a fold at $x_{0}$. Globally $\Ssec$ consists of the two discs
		$p_{1}=\pm\sqrt{2(1-W)}$ over $\Areg$, glued along the circle
		$r=\sqrt2$, where $\pi_E$ is a fold, and at the vertex.
	\end{counterexample}
	
	What the example shows is exactly
	\[
	\textup{(C)}+\textup{(U)}+\textup{(N)}\;\not\Longrightarrow\;\DE\subseteq\p\Areg .
	\]
	Fibrewise nondegeneracy does not prevent the discriminant from acquiring
	interior points when the least-energy function has a local maximum at the
	prescribed energy. The example says nothing about momentum-convex
	Hamiltonians: here $\p^{2}H/\p p_{2}^{2}=4(q_{2}^{2}+3p_{2}^{2}-1)$ equals
	$-4$ at the origin, so $D^{2}_{pp}H$ is not positive definite, in agreement
	with Corollary~\ref{cor:convexbdry}. For every $E\neq0,1$ the same
	Hamiltonian satisfies (R) and Theorem~\ref{thm:structure} applies.
	
	\begin{remark}[the three nondegenerate cases]\label{rem:morse}
		Let $H$ be smooth with \textup{(C)}, \textup{(U)}, \textup{(N)}, and let
		$b_{0}$ be a nondegenerate critical point of $\Ered$ with
		$\Ered(b_{0})=E$. By the normal form $\Phi=\tfrac12a\,v^{2}+s(b)$ in the
		proof of Theorem~\ref{thm:fold}, with $a>0$ and $s=\Ered-E$, and the Morse
		lemma for $s$, the local picture is one of three.
		\begin{itemize}[leftmargin=2.2em]
			\item \emph{Minimum.} $b_{0}$ is an isolated point of $\Areg$ and of
			$\DE$, and $\Ssec$ is locally the single point $(b_{0},m(b_{0}))$.
			\item \emph{Saddle.} $\DE$ is locally a pair of crossing curves and
			equals $\p\Areg$; $\Ssec$ is locally a cone.
			\item \emph{Maximum.} $b_{0}$ is an interior discriminant point; $\Ssec$
			is locally a cone, as in Counterexample~\ref{cex:hat}.
		\end{itemize}
		In all three cases \textup{(R)} and \textup{(T)} fail; only the third
		violates $\DE=\p\Areg$.
	\end{remark}
	
	\section{Variational reformulation of the fold}\label{sec:variational}
	
	The squared residual gives a variational form of \eqref{eq:constraint} that
	is defined over every base point. It does not resolve the fold: the fold is a
	property of $\pi_E$, and the residual only represents it differently, in a
	way that makes the degeneration quantitative.
	
	\begin{definition}\label{def:F}
		For $b\in\Cred$, $\Fb(p_{1}):=\bigl(\hb(p_{1})-E\bigr)^{2}$.
	\end{definition}
	
	\begin{proposition}[minimizers of the residual]\label{prop:F-exist}
		Assume \textup{(C)}. Then $\Fb$ attains its minimum and
		\[
		\min_{p_{1}}\Fb=\bigl(\Ered(b)-E\bigr)_{+}^{2},
		\]
		so that $\min\Fb=0$ if and only if $b\in\Areg$, and then the minimizers
		are the lifts of $b$. Under \textup{(C)} and \textup{(U)}:
		\begin{enumerate}[label=\textup{(\alph*)},leftmargin=2.4em]
			\item if $\Ered(b)<E$, then $\Fb$ has exactly two global minimizers
			$p^{\pm}(b)$, separated by a strict local maximum at $m(b)$ with value
			$(\Ered(b)-E)^{2}$; it is a double well and is not convex;
			\item if $\Ered(b)=E$, the unique minimizer is $m(b)$;
			\item if $\Ered(b)>E$, the unique minimizer is $m(b)$, which is not a
			lift.
		\end{enumerate}
	\end{proposition}
	
	\begin{proof}
		By (C), $\Fb\to+\infty$, so the minimum is attained. If $b\in\Areg$ it is
		$0$, attained exactly at the roots of $\hb=E$. If $\Ered(b)>E$ then
		$\hb-E\ge\Ered(b)-E>0$ and $\Fb$ is minimal where $\hb$ is. Under (U),
		$\Fb'=2(\hb-E)\hb'$. If $\Ered(b)<E$, then $\hb-E<0$ on
		$(p^{-},p^{+})$, so $\Fb'>0$ on $(p^{-},m)$ and $\Fb'<0$ on $(m,p^{+})$,
		which gives (a). Cases (b) and (c) follow from
		Proposition~\ref{prop:mult} and the uniqueness of the minimizer of $\hb$.
	\end{proof}
	
	\begin{remark}[branch selection is a separate rule]\label{rem:branch-rule}
		Over $\{\Ered<E\}$ the minimization returns the two-point set
		$\{p^{-}(b),p^{+}(b)\}$ and does not single out a lift. Choosing a sheet
		is an additional rule. The forward-crossing convention $\dot q_{1}>0$
		selects $p^{+}$ (Corollary~\ref{cor:sheets}(d)); it is a convention on
		the direction in which the flow crosses the section, not a consequence of
		minimizing $\Fb$. Under (C) and (U) the set of all minimizers,
		\[
		\Ssec\;\cup\;\{(b,m(b))\mid \Ered(b)>E\},
		\]
		extends the reconstruction correspondence continuously beyond $\Areg$ by
		the graph of $m$.
	\end{remark}
	
	\begin{corollary}[quartic flatness on the discriminant]\label{cor:quartic}
		Let $x_{0}=(b_{0},p_{1}^{c})\in\Crit(\pi_E)$ satisfy \textup{(N)}, and
		let $h_{b_{0}}\in C^{3}$. Then, as $p_{1}\to p_{1}^{c}$,
		\begin{align*}
			h_{b_{0}}(p_{1})-E&=\tfrac12h_{b_{0}}''(p_{1}^{c})(p_{1}-p_{1}^{c})^{2}
			+O\bigl(|p_{1}-p_{1}^{c}|^{3}\bigr),\\
			F_{b_{0}}(p_{1})&=\tfrac14h_{b_{0}}''(p_{1}^{c})^{2}(p_{1}-p_{1}^{c})^{4}
			+O\bigl(|p_{1}-p_{1}^{c}|^{5}\bigr).
		\end{align*}
		The double root of $\hb-E$ becomes a quartic minimum of $\Fb$: the residual
		squares the order of the degeneration.
	\end{corollary}
	
	\begin{corollary}[scalings near the discriminant]\label{cor:fingerprint}
		Assume \textup{(C)} and \textup{(U)}, fix $b$ with \textup{(N)} at
		$m=m(b)$ and $\hb\in C^{3}$, and let the energy decrease to $\Ered(b)$:
		$\Delta:=E-\Ered(b)\to0^{+}$. Then
		\[
		\begin{aligned}
			p^{\pm}-m
			&=\pm\sqrt{\frac{2\Delta}{\hb''(m)}}+O(\Delta),
			&
			\hb'(p^{\pm})
			&=\pm\sqrt{2\,\hb''(m)\,\Delta}+O(\Delta),
			\\[1mm]
			\Fb''(p^{\pm})
			&=2\,\hb'(p^{\pm})^{2}
			=4\,\hb''(m)\,\Delta+O(\Delta^{3/2}).
		\end{aligned}
		\]
	\end{corollary}
	
	\begin{proof}
		With $w=p_{1}-m$ one has $\hb(p_{1})-\Ered(b)=\tfrac12\hb''(m)w^{2}+O(|w|^{3})$
		and $\hb''(m)>0$. Solving $\hb(p_{1})-\Ered(b)=\Delta$ gives
		$w=\pm\sqrt{2\Delta/\hb''(m)}+O(\Delta)$, and
		$\hb'(p_{1})=\hb''(m)w+O(w^{2})$ gives the second expansion. From
		$\Fb''=2(\hb')^{2}+2(\hb-E)\hb''$ and $\hb(p^{\pm})=E$ one gets
		$\Fb''(p^{\pm})=2\hb'(p^{\pm})^{2}$.
	\end{proof}
	
	\begin{remark}\label{rem:fingerprint}
		The separation of the lifts scales as $\Delta^{1/2}$, the slope
		$\dot q_{1}=\hb'(p^{\pm})$ at which the flow crosses the section also as
		$\Delta^{1/2}$, and the curvature of the residual at its minimizers as
		$\Delta$. The first two are the fold of Theorem~\ref{thm:fold}; the third
		shows that the minimization formulation is not better conditioned than
		root-finding near $\DE$. If $H\in C^{3}$ and (N) holds along $m$, the
		constants are locally uniform in $b$, and the same expansions hold for
		fixed $E$ as $b$ approaches $\DE$ from $\{\Ered<E\}$. These are statements
		about the functions $\hb$ and $\Fb$, not about any algorithm.
	\end{remark}
	
	\begin{definition}[energy deficit]\label{def:deficit}
		$\delta(b):=\sqrt{\min_{p_{1}}\Fb(p_{1})}$.
	\end{definition}
	
	\begin{corollary}\label{cor:deficit}
		Under \textup{(C)}, $\delta=(\Ered-E)_{+}$. It is continuous, vanishes
		exactly on $\Areg$, and outside $\Areg$ equals the amount by which the
		least energy on the leaf exceeds $E$.
	\end{corollary}
	
	\section{Applications to Hamiltonian Poincar\'e sections}
	\label{sec:applications}
	
	Both examples are natural mechanical systems with unit mass matrix, so
	$D^{2}_{pp}H$ is the identity. By Definition~\ref{def:convex},
	Proposition~\ref{prop:convex} and Corollary~\ref{cor:convexbdry}, hypotheses
	(C), (U), (N) hold and $\DE=\p\Areg$ at every energy (for the swing--spring,
	over the base specified in Section~\ref{ssec:ss}); what depends on $E$ is
	only (R), that is, whether $\Ssec$ is smooth and $\pi_E$ is a fold along all
	of $\DE$.
	
	\subsection{H\'enon--Heiles}\label{ssec:hh}
	
	The H\'enon--Heiles system was introduced as a simplified model of the motion
	of a star in an axisymmetric galactic potential, and has become a standard
	example of a nonintegrable Hamiltonian system with two degrees of freedom
	\cite{henon1964,ennsmcguire2001}. Its Hamiltonian is
	\[
	H=\tfrac12(p_{x}^{2}+p_{y}^{2})+V(x,y),\qquad
	V(x,y)=\tfrac12(x^{2}+y^{2})+x^{2}y-\tfrac13y^{3},
	\]
	where $(x,y)$ are configuration coordinates, $(p_{x},p_{y})$ their conjugate
	momenta, the first term is the kinetic energy and $V$ is a harmonic well
	perturbed by cubic terms with threefold symmetry. The potential has a minimum
	$V=0$ at the origin and three saddle points, $(0,1)$ and
	$(\pm\tfrac{\sqrt3}{2},-\tfrac12)$, all at the height
	\[
	\Eesc=\tfrac16 .
	\]
	For $E<\Eesc$ the equipotential $\{V=E\}$ around the origin is a closed curve
	and the motion started in the well is confined to it. At $E=\Eesc$ this curve
	passes through the three saddles, and for $E>\Eesc$ three channels open
	through which orbits can leave the well.
	
	Fixing $E$ restricts the dynamics to the three-dimensional energy surface,
	and the section $\{x=0\}$ gives a two-dimensional picture of it in the
	coordinates $(y,p_{y})$. On this section $p_{x}$ is not a coordinate: it is
	the variable to be reconstructed. In the notation of the paper,
	$q_{1}=x$, $b=(y,p_{y})$,
	\[
	\hb(p_{x})=\tfrac12p_{x}^{2}+U(y,p_{y}),\qquad
	U(y,p_{y})=\tfrac12p_{y}^{2}+\tfrac12y^{2}-\tfrac13y^{3},
	\]
	so that $m\equiv0$, $\Ered=U$, $p^{\pm}=\pm\sqrt{2(E-U)}$,
	$\Areg=\{U\le E\}$ and $\DE=\{U=E\}=\p\Areg$. The critical points of $U$ are
	$(0,0)$, with $U=0$, and $(1,0)$, with $U=\Eesc$; they are the two equilibria
	of the flow that lie on the section, the centre at the origin and the saddle
	$(x,y)=(0,1)$. Hence (R) holds exactly for $E\notin\{0,\Eesc\}$, and for
	those energies $\pi_E$ is a fold-branched double covering of $\Areg$
	(Theorem~\ref{thm:structure}).
	
	For $0<E<\Eesc$ the admissible region has two components
	(Figure~\ref{fig:hh-bifurcation}): a bounded one, $B_E$, containing the
	origin, which is the trace of the potential well on the section, and an
	unbounded one in $\{y>1\}$, beyond the barrier, which orbits in the well do
	not reach. By Corollary~\ref{cor:sheets}(d) the forward section
	$\{x=0,\ p_{x}>0\}$ of the well is the open set $\intr B_E$, with
	$p_{x}=p^{+}(y,p_{y})$. Its boundary $\p B_E$ is a component of $\DE$, and
	Proposition~\ref{prop:flow} can be read off directly: the plane
	$\{x=0,\ p_{x}=0\}$ is invariant, because $\dot x=p_{x}$ and
	$\dot p_{x}=-x(1+2y)$, so the fold curve over $\p B_E$ is itself a periodic
	orbit, the oscillation along the $y$-axis, and the flow is tangent to the
	section along it.
	
	We take $E=\tfrac18<\Eesc$ for the illustration: the orbit remains in the
	bounded potential well, while the nonlinear coupling is already dynamically
	relevant, which makes the section a useful setting in which to display the
	reconstruction. Figure~\ref{fig:hh-poincare} shows forward crossings of
	several orbits. Each orbit is started on the section at a point
	$(y,p_{y})\in\intr B_E$ with the reconstructed momentum $p_{x}=p^{+}$, and a
	point is recorded at every return to $x=0$ with $p_{x}>0$. The
	H\'enon--Heiles section is a classical setting for the study of regular and
	irregular motion \cite{henon1964}; we do not analyse the orbits here and use
	the figure only to display four things: the reconstruction of $p_{x}$, the
	selection of the sheet $p_{x}>0$, the fact that the recorded points lie on
	$\{x=0\}\cap\{H=E\}$, and their position relative to $\Areg$ and $\DE$. For
	the orbit checked in the accompanying script (Section~\ref{sec:impl}), $160$
	crossings are recorded with $|x|<10^{-13}$ and $|H-E|<10^{-9}$.
	
	\begin{figure}[htbp]
		\centering
		\includegraphics[width=0.72\linewidth]{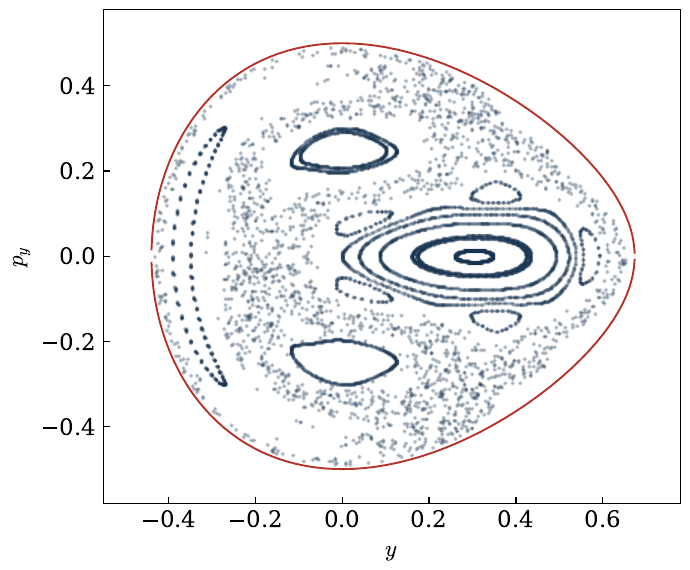}
		\caption{H\'enon--Heiles system at $E=1/8$: forward crossings of the
			section $\{x=0\}$, $p_{x}>0$, in the reduced coordinates $(y,p_{y})$,
			for several initial conditions with $p_{x}=p^{+}(y,p_{y})$. The closed
			curve is $\p B_E$, the bounded component of the discriminant
			$\DE=\{U=E\}=\p\Areg$; all crossings lie in $\{U<E\}$
			(Corollary~\ref{cor:sheets}).}
		\label{fig:hh-poincare}
	\end{figure}
	
	Figure~\ref{fig:hh-poincare} displays the dynamical section obtained after
	reconstructing the missing momentum and selecting the forward-crossing sheet,
	whereas Figure~\ref{fig:hh-bifurcation} displays the energy geometry in the
	reduced variables that determines where such a reconstruction is possible:
	the family $\Areg=\{U\le E\}$ and its boundary $\DE$. As $E$ increases
	towards $\Eesc$, $B_E$ grows. At $E=\Eesc$ it meets the unbounded component
	at the saddle $(1,0)$. This is the saddle case of Remark~\ref{rem:morse}:
	$\DE=\p\Areg$ still holds, as Corollary~\ref{cor:convexbdry} requires, but
	(R) fails, $\Ssec$ has a conical point at the saddle equilibrium, and $\pi_E$
	is a fold everywhere on $\DE$ except there. For $E>\Eesc$ the two components
	have merged. Physically this is the opening of the escape channel through
	the saddle $(0,1)$, the one that intersects the section; by the symmetry of
	$V$ the other two channels open at the same energy, away from the section.
	The other critical value, $E=0$, is the minimum case: $B_{0}$ is the single
	point $(0,0)$, the equilibrium.
	
	\begin{figure}[htbp]
		\centering
		\includegraphics[width=0.72\linewidth]{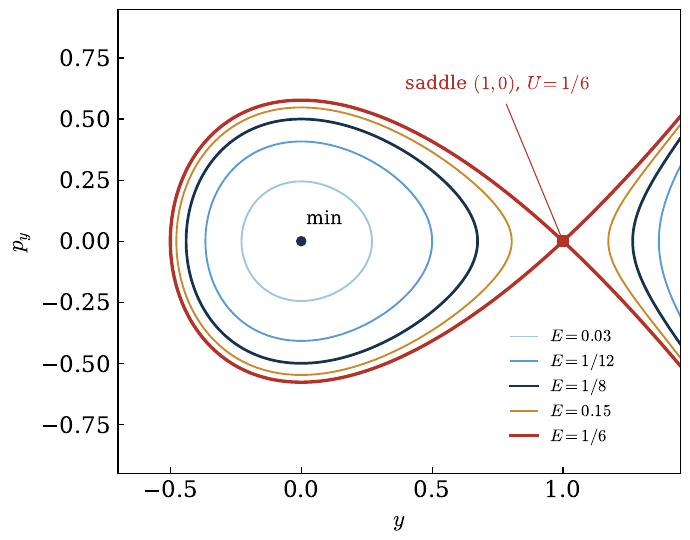}
		\caption{H\'enon--Heiles system: the curves $\DE=\{U=E\}=\p\Areg$ in the
			reduced plane $(y,p_{y})$ for several energies. For $E<\Eesc=1/6$ the
			admissible region $\{U\le E\}$ consists of a bounded component around
			the minimum and an unbounded one to the right of the saddle. At
			$E=\Eesc$ the two meet at the saddle $(1,0)$ of $U$, which is the saddle
			equilibrium $(x,y)=(0,1)$ of the flow, and for $E>\Eesc$ the escape
			channel through it is open.}
		\label{fig:hh-bifurcation}
	\end{figure}
	
	\FloatBarrier
	\subsection{Swing--spring}\label{ssec:ss}
	
	The swing--spring, or elastic pendulum, is a point mass suspended from a
	fixed pivot by a spring and moving in a vertical plane under gravity. It
	combines a pendulum-like swinging motion with the radial oscillation of the
	spring, and the coupling between the two makes it a classical nonlinear
	Hamiltonian system \cite{lowenstein2012}. In nondimensional coordinates
	$(x,z)$, with $z$ vertical, the pivot at $(0,1)$ and the stable equilibrium
	at the origin, the Hamiltonian is
	\[
	H=\tfrac12(p_{x}^{2}+p_{z}^{2})+V(x,z),\qquad
	V(x,z)=z+2\Bigl(\tfrac34-\sqrt{x^{2}+(z-1)^{2}}\Bigr)^{2}-\tfrac18 .
	\]
	Linearizing at the origin gives frequencies $1$ for the swinging mode and
	$2$ for the springing mode.
	
	The potential is smooth except at the pivot, where it is only continuous and
	takes the value $V=2$; the pivot lies on the section $\{x=0\}$. For $E<2$
	the energy surface does not meet it. On the section,
	\[
	V(0,z)=
	\begin{cases}
		2z^{2}, & z\le1,\\[1mm]
		2z^{2}-6z+6, & z\ge1,
	\end{cases}
	\]
	and the second branch is at least $\tfrac32$. We therefore work over the base
	$B=\{z<1\}$ in the sense of Remark~\ref{rem:local-base}, where $H$ is smooth,
	and with $E<\tfrac32$, for which the whole admissible region lies in $B$.
	
	The purpose of this example is that it is a particularly clean instance of
	the convex regime. The momentum Hessian is the identity and coercivity is
	evident, since
	\[
	\hb(p_{x})=\tfrac12p_{x}^{2}+U(z,p_{z}),\qquad
	U(z,p_{z})=\tfrac12p_{z}^{2}+2z^{2}\quad(z<1).
	\]
	Thus $\Ered=U$, and for $0<E<\tfrac32$ the admissible region is the ellipse
	$\Areg=\{\tfrac12p_{z}^{2}+2z^{2}\le E\}$, with $\DE=\p\Areg$. The only
	critical value of $U$ on $B$ is $0$, so (R) holds and $\pi_E$ is a
	fold-branched double covering of $\Areg$ (Theorem~\ref{thm:structure}), with
	$p^{\pm}=\pm\sqrt{2(E-U)}$. As in the previous example the plane
	$\{x=0,\ p_{x}=0\}$ is invariant, since $\dot p_{x}$ is proportional to $x$,
	and the fold curve is the periodic orbit of purely vertical oscillation.
	
	Figure~\ref{fig:ss-poincare} shows the forward section at $E=0.84$, obtained
	exactly as before: $p_{x}$ is reconstructed on $\{x=0\}$ as $p^{+}$, and
	crossings with $p_{x}>0$ are recorded in the coordinates $(z,p_{z})$. The
	same covering-and-fold structure is reproduced, which shows that the
	mechanism does not depend on the particular form of the H\'enon--Heiles
	potential; here the potential is not polynomial and the admissible region is
	an ellipse. For the orbit checked in the script, $80$ crossings are recorded
	with $|x|<10^{-13}$ and $|H-E|<10^{-8}$.
	
	\begin{figure}[htbp]
		\centering
		\includegraphics[width=0.6\linewidth]{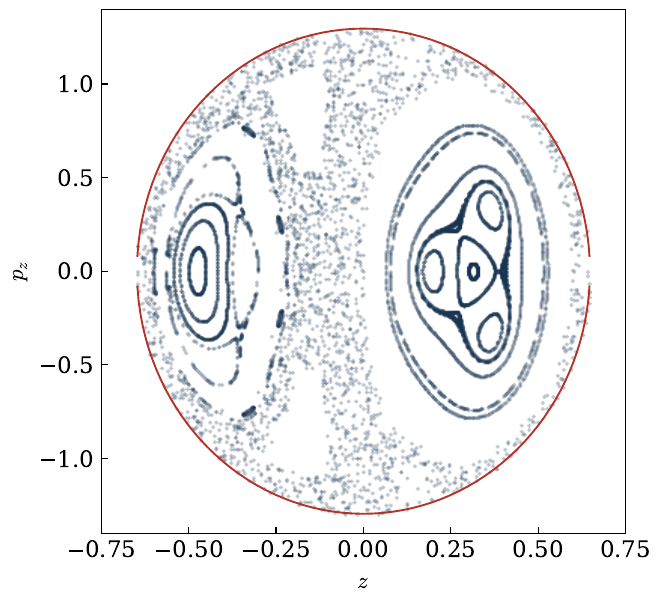}
		\caption{Swing--spring system at $E=0.84$: forward crossings of the
			section $\{x=0\}$, $p_{x}>0$, in the reduced coordinates $(z,p_{z})$,
			with $p_{x}=p^{+}(z,p_{z})$ at the initial points. The closed curve is
			the discriminant $\DE=\{U=E\}=\p\Areg$, an ellipse.}
		\label{fig:ss-poincare}
	\end{figure}
	
	\FloatBarrier
	\section{Conclusions and perspectives}\label{sec:conclusion}
	
	The recovery of the missing momentum on a Poincar\'e section is governed by
	the projection $\pi_E\colon\Ssec\to\Cred$ attached to the coisotropic
	reduction of the section. Its critical points are the tangencies of the flow
	with the section, and the hypotheses behind its properties separate as
	follows:
	\begin{align*}
		\textup{(C)}\;&\Longrightarrow\;\pi_E \text{ proper, finite-sheeted covering of its image off } \DE;\\
		&\phantom{\Longrightarrow}\;\ \Areg=\{\Ered\le E\},\quad
		\p\Areg\subseteq\DE;\\
		\textup{(C)}+\textup{(U)}\;&\Longrightarrow\;\text{two sheets } p^{\pm}
		\text{ over } \{\Ered<E\},\quad \DE=\{\Ered=E\};\\
		&\phantom{\Longrightarrow}\;\ \{\dot q_{1}>0\}\cap\Ssec=
		\operatorname{graph}(p^{+});\\
		\textup{(N)}+\textup{(T)}\;&\Longrightarrow\;\text{Whitney fold at the
			point};\\
		\textup{(C)}+\textup{(U)}+\textup{(N)}\;&\Longrightarrow\;\Ered\in C^{2},
		\quad \textup{(R)} \text{ for almost every } E;\\
		\textup{(C)}+\textup{(U)}\;&\Longrightarrow\;\bigl[\,\DE=\p\Areg\iff
		\Ered \text{ has no local maximum on the level } E\,\bigr];\\
		\textup{(C)}+D^{2}_{pp}H\succ0\;&\Longrightarrow\;\textup{(U)},\
		\textup{(N)},\ \text{ and } \DE=\p\Areg \text{ for every } E .
	\end{align*}
	Regularity (R) of the level, equivalently regularity of the value $E$ for
	$H|_{C}$, equivalently (T) along the whole critical locus, is a sufficient
	condition for $\DE=\p\Areg$ and is exactly what is needed, beyond (C), (U),
	(N), for $\pi_E$ to be a fold-branched double covering
	(Theorem~\ref{thm:structure}). ``Branched covering'' is used throughout as a
	shorthand for this covering-plus-fold structure
	(Definition~\ref{def:branched}), not for the topological notion with a
	branch set of codimension two.
	
	The identification of the discriminant with the boundary of the admissible
	region can fail only at interior discriminant points, which are the local
	maxima of the least-energy function at the prescribed energy.
	Counterexample~\ref{cex:hat} shows that nondegenerate coercive fibres do not
	exclude them; Corollary~\ref{cor:convexbdry} shows that a positive-definite
	momentum Hessian with coercive fibres does, at every energy. For natural
	mechanical systems, therefore, the boundary of the admissible region on the
	section is always the projection of the tangency set of the flow, and the
	only thing that can go wrong at special energies is the smoothness of
	$\Ssec$ and the fold, as at the saddle energy of the H\'enon--Heiles system.
	When (N) fails, as for $p_{1}^{4}$, the projection is more degenerate than a
	fold (Counterexample~\ref{cex:nd}); when (C) fails, positivity of the
	momentum Hessian gives neither (U) nor the fibre count
	(Counterexample~\ref{cex:coer}).
	
	The least-energy function determines the admissible region but not the
	correspondence (Proposition~\ref{prop:lossy}). The squared residual is a
	variational reformulation of the recovery, defined over every base point,
	with two global minimizers over $\{\Ered<E\}$ and a curvature that vanishes
	linearly at the discriminant (Proposition~\ref{prop:F-exist},
	Corollary~\ref{cor:fingerprint}); it makes the degeneration explicit and
	leaves the choice of sheet to a separate convention.
	
	Three directions remain open. First, for $n$ degrees of freedom and a
	hypersurface section the characteristic leaves are still one-dimensional,
	a single momentum $p_{1}$ is reconstructed, and $\pi_E$ is a map between
	manifolds of dimension $2n-2$ with one-dimensional kernel at most; where (N)
	fails one expects the singularities beyond the fold to be governed by the
	order of vanishing of the derivatives of $\hb$, and we have not studied
	them. Second, the critical energies of $\Ered$, where (R) fails, are the
	only candidates for changes in the topology of $\Areg$ and deserve a
	separate study. Third, above escape thresholds the return map must be
	replaced by exit maps, while the reconstruction correspondence remains
	defined.
	
	\section{Computational reproducibility}\label{sec:impl}
	
	The results of this paper are proved analytically. A \textsc{Wolfram
		Language} script and a notebook that calls it make the symbolic statements
	about the examples and the two numerical sections reproducible; they are
	checks on examples and do not replace any proof. The script is the file
	\begin{center}
		\texttt{ReconstructionCoisotropicReduction.wl}
	\end{center}
	Loading it defines its functions, and the routine
	\texttt{RunReproducibility[]} runs all checks and returns a
	\textsc{TestReport} object.
	
	The script computes the discriminant of Definition~\ref{def:disc} by
	eliminating $p_{1}$ from $\{\hb=E,\ \p_{p_{1}}\hb=0\}$ and the admissible
	region of Definition~\ref{def:adm} by quantifier elimination of
	$\exists\,p_{1}\,(\hb=E)$, and compares both with $\{\Ered=E\}$ and
	$\{\Ered\le E\}$ for the H\'enon--Heiles fibre function. It verifies the
	count of Proposition~\ref{prop:mult} on $\tfrac12p_{1}^{2}+w$, and on the
	same family it obtains the minimizers of the residual and the deficit
	$\delta=\sqrt{\min\Fb}$ by symbolic minimization
	(Proposition~\ref{prop:F-exist}, Corollary~\ref{cor:deficit}). For
	Counterexample~\ref{cex:hat} at $E=1$ it returns the disc $r\le\sqrt2$ as
	admissible region, finds the origin in the discriminant, and checks that the
	origin is a local maximum of $\Ered$ and that $\p^{2}H/\p p_{2}^{2}<0$
	there. It tests (C), (U) and (N) on the fibre functions $\tfrac12p_{1}^{2}$,
	$e^{p_{1}}$ and $p_{1}^{4}$ of Section~\ref{ssec:convex}, the test of (U)
	being the existence of a single critical point with the sign change of
	$\hb'$. It computes the expansions of Corollary~\ref{cor:quartic} and the
	leading terms of Corollary~\ref{cor:fingerprint} by series inversion on a
	fibre function with a cubic term. It applies the forward-crossing rule of
	Corollary~\ref{cor:sheets}(d), checks that the momentum Hessians of the two
	examples are the identity and that the critical point $(1,0)$ of $U$ on the
	level $\Eesc$ is a saddle, and integrates one orbit of each system of
	Section~\ref{sec:applications}, recording the section crossings and the
	conservation of energy quoted there. The figures were rendered with a
	separate plotting script from the same equations, the same initial rule
	$p_{x}=p^{+}$ and the same crossing convention.
	
	An interactive point-and-click explorer for sections of
	two-degree-of-freedom systems is available as the earlier tool
	\texttt{ClickPoincarePlot2D} \cite{chanlopez2024wolfram}. The present
	results do not depend on it.
	
	\section*{Acknowledgements}
	E.~Chan-L\'opez acknowledges support from SECIHTI through the ``Estancias
	Posdoctorales por M\'exico'' program (CVU 422090).
	
	\section*{Ethics declarations}
	
	\subsection*{Conflict of interest}
	The author declares no conflict of interest.
	
	\subsection*{Ethical approval}
	Not applicable.
	


\begin{thebibliography}{20}
		
		\bibitem{abraham1978}
		R.~Abraham and J.\,E.~Marsden,
		\emph{Foundations of Mechanics}, 2nd ed.,
		Benjamin/Cummings, Reading, MA, 1978.
		
		\bibitem{arnold}
		V.\,I.~Arnold,
		\emph{Mathematical Methods of Classical Mechanics}, 2nd ed.,
		Graduate Texts in Mathematics, vol.~60, Springer, New York, 1989.
		\newblock \doi{10.1007/978-1-4757-2063-1}
		
		\bibitem{bolsinov1996}
		A.~Bolsinov, H.\,R.~Dullin, and A.~Wittek,
		Topology of energy surfaces and existence of transversal Poincar\'e sections,
		\emph{J.\ Phys.\ A: Math.\ Gen.} \textbf{29} (1996), no.~16, 4977--4985.
		\newblock \doi{10.1088/0305-4470/29/16/019}
		
		\bibitem{chanlopez2024wolfram}
		E.~Chan-L\'opez, H.~Argote Morales, W.~Szumiński, and V.~Castellanos,
		\href{https://resources.wolframcloud.com/FunctionRepository/resources/ClickPoincarePlot2D/}
		{\texttt{ClickPoincarePlot2D}},
		\emph{Create an interactive view of Poincar\'e sections of a two-degree-of-freedom Hamiltonian system}, Wolfram Function Repository, 2023.
		
		\bibitem{dullin1995}
		H.\,R.~Dullin and A.~Wittek,
		Complete Poincar\'e sections and tangent sets,
		\emph{J.\ Phys.\ A: Math.\ Gen.} \textbf{28} (1995), no.~24, 7157--7180.
		\newblock \doi{10.1088/0305-4470/28/24/015}
		
		\bibitem{elnabulsi2026}
		R.\,A.~El-Nabulsi, W.~Anukool, C.~Thangaraj, and R.~Valarmathi,
		Stability and chaos in nonstandard Hamiltonian planetary dynamics,
		\emph{Chaos Solitons Fractals} \textbf{202} (2026), 117410.
		\newblock \doi{10.1016/j.chaos.2025.117410}
		
		\bibitem{ennsmcguire2001}
		R.\,H.~Enns and G.\,C.~McGuire,
		\emph{Nonlinear Physics with Mathematica for Scientists and Engineers},
		Birkh\"auser, Boston, MA, 2001.
		\newblock \doi{10.1007/978-1-4612-0211-0}
		
		\bibitem{gashenenko2004}
		I.\,N.~Gashenenko and P.\,H.~Richter,
		Enveloping surfaces and admissible velocities of heavy rigid bodies,
		\emph{Int.\ J.\ Bifur.\ Chaos Appl.\ Sci.\ Engrg.} \textbf{14} (2004), no.~8,
		2525--2553.
		\newblock \doi{10.1142/S021812740401103X}
		
		\bibitem{golubitsky1973}
		M.~Golubitsky and V.~Guillemin,
		\emph{Stable Mappings and Their Singularities},
		Graduate Texts in Mathematics, vol.~14, Springer, New York, 1973.
		
		\bibitem{henon}
		M.~H\'enon,
		On the numerical computation of Poincar\'e maps,
		\emph{Physica D} \textbf{5} (1982), no.~2--3, 412--414.
		\newblock \doi{10.1016/0167-2789(82)90034-3}
		
		\bibitem{henon1964}
		M.~H\'enon and C.~Heiles,
		The applicability of the third integral of motion: some numerical experiments,
		\emph{Astronom.\ J.} \textbf{69} (1964), 73--79.
		\newblock \doi{10.1086/109234}
		
		\bibitem{jose1998}
		J.\,V.~Jos\'e and E.\,J.~Saletan,
		\emph{Classical Dynamics: A Contemporary Approach},
		Cambridge University Press, Cambridge, 1998.
		
		\bibitem{lowenstein2012}
		J.\,H.~Lowenstein,
		\emph{Essentials of Hamiltonian Dynamics},
		Cambridge University Press, Cambridge, 2012.
		\newblock \doi{10.1017/CBO9780511793721}
		
		\bibitem{meyer}
		K.\,R.~Meyer, G.\,R.~Hall, and D.~Offin,
		\emph{Introduction to Hamiltonian Dynamical Systems and the $N$-Body Problem},
		2nd ed., Applied Mathematical Sciences, vol.~90, Springer, New York, 2009.
		\newblock \doi{10.1007/978-0-387-09724-4}
		
		\bibitem{milnor1965}
		J.\,W.~Milnor,
		\emph{Topology from the Differentiable Viewpoint},
		University Press of Virginia, Charlottesville, 1965.
		
		\bibitem{salasnich2026}
		L.~Salasnich and F.~Sattin,
		Geometric investigation of chaos unfolding in Hamiltonian systems,
		\emph{Chaos Solitons Fractals} \textbf{208} (2026), 118259.
		\newblock \doi{10.1016/j.chaos.2026.118259}
		
		\bibitem{schmidt2009}
		S.~Schmidt, H.\,R.~Dullin, and P.\,H.~Richter,
		A Poincar\'e section for the general heavy rigid body,
		\emph{SIAM J.\ Appl.\ Dyn.\ Syst.} \textbf{8} (2009), no.~1, 371--389.
		\newblock \doi{10.1137/080719029}
		
		\bibitem{wang2026reconstruction}
		L.~Wang, Y.~Su, B.~Wang, M.~Yuan, and W.~Xu,
		Model reconstruction for nonlinear dynamical systems on Poincar\'e section,
		\emph{Chaos Solitons Fractals} \textbf{205} (2026), 117854.
		\newblock \doi{10.1016/j.chaos.2025.117854}
		
		\bibitem{zhangdai2025}
		Z.~Zhang and L.~Dai,
		Secondary Poincar\'e section for characterizing dynamical behaviours of
		nonlinear systems,
		\emph{Chaos Solitons Fractals} \textbf{199} (2025), 116849.
		\newblock \doi{10.1016/j.chaos.2025.116849}
		
	\end{thebibliography}
\end{document}